\documentclass{amsart}

\usepackage{amsmath, amsthm, amssymb, amstext, amsfonts}
\usepackage{enumerate}
\usepackage{graphicx}
\usepackage[applemac]{inputenc}

\theoremstyle{plain}

\newtheorem{thm}{Theorem}[section]

\newtheorem{prop}[thm]{Proposition}
\newtheorem{rem}[thm]{Remark}

\newtheorem*{prob*}{Problem}

\theoremstyle{definition}

\newcommand{\cV}{\ensuremath{\mathcal{V}}}

\newcommand{\sm}{\ensuremath{\smallsetminus}}

\newcommand{\Aut}{\textnormal{Aut}}

\newcommand{\sub}{\subseteq}

\def\td{tree-decom\-po\-si\-tion}

\newcommand{\comment}[1]{}

\newcommand{\nat}{{\mathbb N}}
\newcommand{\real}{{\mathbb R}}

\newcommand{\BF}{\ensuremath{\mathcal B}}
\newcommand{\CF}{\ensuremath{\mathcal C}}
\newcommand{\DF}{\ensuremath{\mathcal D}}
\newcommand{\EF}{\ensuremath{\mathcal E}}

\newcommand{\HF}{\ensuremath{\mathcal H}}

\def\?#1{\vadjust{\vbox to 0pt{\vss\vskip-8pt\leftline{%
     \llap{\hbox{\vbox{\pretolerance=-1
     \doublehyphendemerits=0\finalhyphendemerits=0
     \hsize16truemm\tolerance=10000\small
     \lineskip=0pt\lineskiplimit=0pt
     \rightskip=0pt plus16truemm\baselineskip8pt\noindent
     \hskip0pt        %(without this, the first word is never hyphenated!)
     #1\endgraf}\hskip7truemm}}}\vss}}}

\newenvironment{txteq*}
  {
    \begin{equation*}
    \begin{minipage}[c]{0.85\textwidth} % set width to 0.9 x textwidth
    \em                                % switch on emph
  }
  {\end{minipage}\end{equation*}\ignorespacesafterend}

\begin{document}

\title{Generating the cycle space of planar graphs}
\author{Matthias Hamann}%\bigskip \\Fachbereich Mathematik\\Universit\"at Hamburg}
%\address{Matthias Hamann, Department of Mathematics, University of Hamburg, Bundes\-stra\ss e~55, 20146 Hamburg, Germany}%, tel +49 40 428385133, email {\tt matthias.hamann@math.uni-hamburg.de}}
\date{}%\today}
\maketitle

\begin{abstract}
We prove that the cycle space of every planar finitely separable $3$-connected graph $G$ is generated by some $\Aut(G)$-invariant nested set of cycles.
We also discuss the situation in the case of smaller connectivity.
\end{abstract}

\section{Introduction}\label{sec_Intro}

Dicks and Dunwoody~\cite{DiDu-GroupsGraphs} showed that the cut space of every connected graph is generated by some nested set of cuts that is invariant under the automorphisms of the graph, where a set of cuts is \emph{nested} if no cut separates any two edges of any other cut.
Recently, Dunwoody~\cite{D-Networks} strengthened this by showing that there is also a \emph{canonical} such set.
This means roughly that no choices were made during the construction of this set or, more precisely, that the construction of the set commutes with graph isomorphisms.

In this note, we present a similar result for the cycle space of a planar graph.
Roughly speaking, two cycles in a planar graph are \emph{nested} if their embeddings in the plane do not cross (see Section~\ref{sec_def} for a precise definition).
Similarly as above, a set of cycles is \emph{canonical} if one of its constructions commutes with graph isomorphisms.\footnote{Note that for any graph $G$ every canonical set of cuts and every canonical set of cycles is invariant under the automorphisms of~$G$.}
We call a graph \emph{finitely separable} if
\begin{equation}\label{eq_FinSep}
\begin{minipage}[t]{0.9\textwidth}
\em
there are only finitely many pairwise internally disjoint paths between any two vertices.
\end{minipage}
\end{equation}
With these definitions, we are able to state our main theorem:

\begin{thm}\label{thm_main}
Every planar finitely separable $3$-connected graph $G$ has a canonical nested set of cycles that generates its cycles space.
\end{thm}

Note that Theorem~\ref{thm_main} is rather trivial for finite graphs: the face boundaries form a nested set of cycles that generates the cycle space and is canonical, as  Whitney~\cite{whitney_congruent_1932} proved that $3$-connected graphs have a unique embedding in the plane (cf.\ Theorem~\ref{thm_Whitney}).\footnote{The same is proof holds for \emph{VAP-free} planar graphs.
These are graphs that admit an embedding in the plane without any accumulation point of the vertices.}

Theorem~\ref{thm_main} can be seen as a dual version of Dunwoody's theorem~\cite{D-Networks}, and indeed, our strategy to prove our main theorem is to deduce it from the fact that Dunwoody's theorem holds for its dual graph.
Of course, it is necessary to have a (unique) dual graph.
Duals of infinite planar graphs are not straightforward analogues of the finite ones, see Thomassen~\cite{T-Duality}.
But Bruhn and Diestel~\cite{BD-Duality} showed that the additional assumption that the graphs are finitely separable implies the existence of unique duals in the $3$-connected case.

\begin{comment}{An example of a planar finitely separable graph of connectivity~$2$ which does not have an $\Aut(G)$-invariant nested set of cycles that generate its cycles space is the following: take two vertices together with two paths of length~$2$ and two paths of length~$3$ between them such that the paths are internally disjoint.
The resulting graph has an embedding in the plane such that its cycle space is not generated by any $\Aut(G)$-invariant nested set of cycles and one embedding such that its cycle space is generated by some such set.}\end{comment}

Theorem~\ref{thm_main} fails for graphs that have lower connectivity~-- regardless which embedding in the plane is considered:
take a graph~$G$ consisting of two vertices $x,y$ and four distinct paths of length~$2$ between them.
Obviously, this graph is planar.
As all cycles in~$G$, each of which has length~$4$, lie in the same $\Aut(G)$-orbit, all of them must be chosen in any canonical set that generates the cycle space.
But it is easy to find two crossing cycles in~$G$.

Even though Theorem~\ref{thm_main} fails for graphs of lower connectivity, it is possible to extend the graph canonically to a graph~$G'$ with equal connectivity such that $G'$ has a canonical nested set of cycles generating its cycle space.
We will discuss this in Section~\ref{sec_LowerCon}.

\begin{comment}{
Let us look at the assumption that the graph is finitely separable once more.
We shall use this assumption when we look at the dual graph as Thomassen~\cite{T-Duality80} showed that any graph with a dual must satisfy this assumption.
Besides our proof here, a second method to prove our main theorem is by immitating the dualised version of the proof of~\cite[Theorem 2.20]{DiDu-GroupsGraphs}, but not using duality directly.
Unfortunately, this proof needs the assumption that the graphs are finitely separable, too.
It is not clear whether Theorem~\ref{thm_main} remains true if we omit that assumption:
}\end{comment}
It is not clear whether Theorem~\ref{thm_main} remains true if we omit the assumption of finite separability:

\begin{prob*}
Does every planar $3$-connected graph $G$ have a canonical nested set of cycles that generates its cycles space?
\end{prob*}

Our main theorem has applications for infinite graphs and infinite groups: in~\cite{H-Accessibility}, the author uses it to show that locally finite quasi-transitive planar graphs are accessible, and in~\cite{GH-PlanarGroups}, Georgakopoulos and the author apply it to obtain planar presentations for planar groups.
(\emph{Planar presentations} of planar groups are presentations that directly tell that the group has a planar Cayley graph.)

\section{Preliminaries}\label{sec_def}

Throughout the paper, a graph may have loops and multiple edges (what is usually called a multigraph).\footnote{Note that all cited results that we need for our main result are valid for this general notion of graphs. Most of them are stated in this way, e.g.\ the result of Dicks and Dunwoody. The only result not stated in this way is Theorem~\ref{thm_DD}, but its proof can easily be adapted to multigraphs. (Note that its forerunner~\cite{DiDu-GroupsGraphs} is stated for multigraphs.)}
Let $G$ be a graph.
A one-way infinite path is a \emph{ray}. Two rays are \emph{equivalent} if they lie eventually in the same component of $G-S$ for every finite $S\sub V(G)$.
This is an equivalence relation whose equivalence classes are the \emph{ends} of~$G$.
By $\Omega(G)$ we denote the set of ends of~$G$.

Now we define a basis of a topology on $G\cup\Omega(G)$.
In order to do so, we view every edge as an isometric copy of the unit interval.
For every $v\in V(G)$ and every $n\in\nat$, the set of all points on edges incident with~$v$ that have distance less than $1/n$ to~$v$ are open.
For every end $\omega$ and every finite vertex set~$S$, let $C(S,\omega)$ be the component of $G-S$ that contains all rays in~$\omega$ eventually and let $\hat{C}(S,\omega)$ be $C(S,\omega)$ together with all ends of~$G$ that contains rays of $C(S,\omega)$.
The set $\hat{C}(S,\omega)$ together with the interior points of the edges between $S$ and $C(S,\omega)$ is open.
We denote by $|G|$ the topological space on $G\cup\Omega(G)$ defined by all these open sets.

A vertex $v$ \emph{dominates} an end $\omega$ if for some ray~$R$ in~$\omega$ there are infinitely many $v$--$R$ paths in~$G$ that pairwise meet only in~$v$.
We call the end~$\omega$ \emph{dominated}.
We define $\tilde G$ to be the quotient space obtained from $|G|$ by identifying every dominated end with its dominating vertices.
Note that for finitely separable graphs every end is identified with at most one vertex and that $\tilde G$ and $|G|$ coincide if~$G$ is locally finite.

A \emph{circle} is a homeomorphic copy of the unit circle in~$\tilde G$ and its \emph{circuit} is the set of edges that it contains.
The \emph{sum} (over $\mathbb F_2$) of finitely many cycles $(C_i)_{i\in I}$ in~$G$ is the set of those edges that occur in an odd number of~$C_i$.
The set of all sums of cycles forms a vector space over $\mathbb F_2$, the \emph{cycle space} $\CF(G)$ of~$G$.
For every $n\in\nat$, let $\CF_n(G)$ be the subspace of the cycle space~$\CF$ that is generated by the cycles of length at most~$n$.
So we have $\CF(G)=\bigcup_{n\in\nat}\CF_n(G)$.
The automorphisms of~$G$ act canonically on~$\CF(G)$.

A \emph{cut} is a subset $B$ of $E(G)$ that is, for some bipartition $\{X,Y\}$ of~$V(G)$, the set of all those edges that have one of its incident vertices in~$X$ and the other in~$Y$.
We call two cuts $B_1$ and $B_2$ \emph{nested} if for bipartition $\{X_1,Y_1\}$ and $\{X_2,Y_2\}$ that correspond to $B_1$ and~$B_2$, respectively, one of the following holds:
\[
X_1\sub X_2,\qquad X_1\sub Y_1,\qquad X_1\supseteq X_2,\text{ or}\qquad X_1\supseteq Y_1.
\]
A cut $B$ is \emph{tight} if $G-B$ consists of two components.
The \emph{sum} (over $\mathbb F_2$) of finitely many cuts in~$G$ is the set of those edges that occur in an odd number of these cuts.
The set of all sums of finite cuts form a vector space over $\mathbb F_2$, the \emph{cut space} $\BF(G)$ of~$G$.
For $n\in\nat$ let $\BF_n(G)$ be the subspace of $\BF(G)$ induced by the tight cuts of size at most~$n$.
So we have $\BF(G)=\bigcup_{n\in\nat}\BF_n(G)$.
Note that no element of~$\BF(G)$ is infinite and that $\Aut(G)$ acts canonically on~$\BF(G)$.

\medskip

Let $G$ and $G^*$ be graphs such that $G$ is finitely separable and let
\[
^*\colon E(G)\to E(G^*)
\]
be a bijection.
We call $G^*$ the \emph{dual} of~$G$ if for every (finite or infinite) $F\sub E(G)$ the following holds:
\begin{equation}\label{eq_Dual}
\begin{minipage}[t]{0.9\textwidth}
\em
$F$ is a circuit in~$G$ if and only if $\{f^*\mid f\in F\}$ is a tight non-empty cut.
\end{minipage}
\end{equation}

\begin{thm}{\cite[Corollary 3.5]{BD-Duality}}\label{thm_UniqueDual}
Every planar finitely separable $3$-connected graph has a unique dual.\qed
\end{thm}

\begin{thm}{\cite[Theorem 4.5]{T-Duality}}
Let $G$ and $G^*$ be dual graphs.
Then $G$ is $3$-connec\-ted if and only if $G^*$ is $3$-connected.\qed
\end{thm}

A \emph{face} of a planar embedding $\varphi\colon G\to\real^2$ is a component of $\real^2\sm \varphi(G)$.
The boundary of a face~$F$ is the set of vertices and edges of~$G$ that are mapped by~$\varphi$ to the closure of~$F$.
A path in~$G$ is \emph{facial} if it is contained in the closure of a face.
The following theorem is due to Whitney \cite[Theorem 11]{whitney_congruent_1932} for finite graphs and due to Imrich~\cite{ImWhi} for infinite ones.

\begin{thm}\cite{ImWhi}\label{thm_Whitney}
Let $G$ be a 3-connected graph embedded in the sphere. Then every automorphism of $G$ maps each facial path to a facial path.

In particular, every automorphism of~$G$ extends to a homeomorphism of the sphere.\qed
\end{thm}

So every planar $3$-connected graph has, basically, a unique planar embedding into the plane.
In the remainder of this paper we always assume this implicitly if we talk about planar embeddings of planar $3$-connected graphs.

Duals of finite graphs, can be easily found using the plane, which we will just recall here, as we will use this well-known fact (see e.g.\ \cite[Section 4.6]{GelbesBuch}).

\begin{rem}\label{rem_GeoAbstrDual}
Let $G$ be a finite planar graph with planar embedding $\varphi\colon G\to\real^2$.
Let $G'$ be the graph whose vertices are the faces of~$\varphi(G)$ and where two vertices are adjacent if their boundaries contain a common edge of~$G$.
Then $G'$ is the dual of~$G$.
\end{rem}

For a planar graph~$G$ with planar embedding $\varphi\colon G\to\real^2$, we call a set $\DF$ of cycles \emph{nested} if for no $C,D\in\DF$ both faces of $\real^2\sm \varphi(C)$ contain vertices of~$D$.
So $\varphi(D)$ lies in one face of~$\real^2\sm \varphi(C)$ together with its boundary~$\varphi(C)$.
We call two cycles $C_1,C_2$ \emph{nested} if $\{C_1,C_2\}$ is nested.

\section{Planar $3$-connected graphs}

Our main tool in the proof of our main theorem will be the following theorem:

\begin{thm}{\cite[Lemma~3.2]{D-Networks}}\label{thm_DD}
If $G$ is a connected graph, then there is a sequence $\EF_1\sub\EF_2\sub\ldots$ of subsets of $\BF(G)$ such that each $\EF_n$ is a canonical nested set of tight cuts of order at most~$n$ that generates~$\BF_n(G)$.\qed
\end{thm}

Basically, our proof goes as follows.
We consider the dual graph $G^*$ of~$G$ and apply Theorem~\ref{thm_DD} to~$G^*$ to find canonical nested sets of cuts.
These sets then will define us our sequence of canonical nested sets of cycles in~$G$.
In preparation of that proof, we shall prove a proposition first.

\begin{prop}\label{prop_NestedInDual}
Let $G$ be a planar finitely separable $3$-connected graph and let $G^*$ be its dual.
If~$C_1$ and~$C_2$ are cycles in~$G$ such that the cuts $C_i^*:=\{e^*\mid e\in C_i\}$ are nested, then $C_1$ and $C_2$ are nested.
\end{prop}

\begin{proof}
Let $\varphi\colon G\to\real^2$ be a planar embedding of~$G$.
Let us suppose that $C_1$ and~$C_2$ are not nested.
Let $H\sub G^*$ be a finite connected subgraph that contains $C_1^*$ and~$C_2^*$ as tight cuts.
Then $H^*$ contains the two cycles $C_1$ and~$C_2$.
For $i=1,2$, let $A_i$ be the unbounded face of $\real^2\sm\varphi(C_i)$ and let $B_i$ be its bounded face.
As $C_1$ and~$C_2$ are not nested, each of $A_1$ and $B_1$ contains edges from~$C_2$.
Due to Remark~\ref{rem_GeoAbstrDual}, the tight cut $C_1^*$ separates edges from $C_2^*$ in~$H$.
So each of the two components of $G^*-C_1^*$ contains vertices of each of the two components of $G^*-C_2^*$.
Thus, $C_1^*$ and $C_2^*$ are not nested in~$H$ and hence neither in~$G^*$.
\end{proof}

Now we are able to prove Theorem~\ref{thm_main2}, a sharpend version of Theorem~\ref{thm_main}.

\begin{thm}\label{thm_main2}
For every planar finitely separable $3$-connected graph $G$ and all $n\in\nat$ there exists a canonical nested set $\DF_n$ of cycles of length at most~$n$ such that $\DF_n$ generates $\CF_n(G)$.

In addition, we may choose the sets $\DF_n$ such that $\DF_{n-1}\sub\DF_n$.
\end{thm}

Note that Theorem~\ref{thm_main} is a direct consequence of Theorem~\ref{thm_main2}: just take the set $\bigcup_{n\in\nat}\DF_n$.
Obviously, this is canonical and nested and it generates the whole cycle space of~$G$.

\begin{proof}[Proof of Theorem~\ref{thm_main2}]
Let $G$ be a planar finitely separable $3$-connected graph with planar embedding $\varphi\colon G\to\real^2$.
Due to Theorem~\ref{thm_UniqueDual}, it has a unique dual~$G^*$, which is $3$-connected due to Thomassen~\cite[Theorem~4.5]{T-Duality}.
So we can apply Theorem~\ref{thm_DD} and obtain a sequence $\EF_1\sub\EF_2\sub\ldots$ of subsets of $\BF(G^*)$ such that each $\EF_n$ is a canonical nested set of tight cuts of order at most~$n$ that generates $\BF_n(G^*)$.

We shall prove that the sequence of the sets
\[
\DF_n:=\{B^*\mid B\in\EF_n\}
\]
with $B^*:=\{e^*\mid e\in B\}$ satisfies the assertion.
Note that $\DF_n$ is a set of cycles of length at most~$n$ as $\EF_n$ is a set of finite tight cuts of size at most~$n$.
The set $\DF_n$ is nested due to Proposition~\ref{prop_NestedInDual} and it is canonical, as $\EF_n$ is canonical and mapping this into the edge set of the duel $G^*$ commutes with graph isomorphisms (for $3$-connected graphs).
As $\EF_n$ generates $\BF_n(G^*)$, the set $\DF_n$ generates $\CF_n(G)$:
it suffices to show that every cycle $C$ of length at most~$n$ is generated by $\DF_n$, but as the tight cut $C^*$ is the sum of finitely many elements of~$\EF_n$, the dual sets in~$E(G)$ of those summands sum to~$C$.
This finishes the proof.
\end{proof}

\section{Planar graphs of small connectivity}\label{sec_LowerCon}

Tutte~\cite{Tutte} proved for finite graphs that $2$-connected graphs can be decomposed into `$3$-connected parts' and cycles.
Later, his result was extended by Droms et al.~\cite{DSS-Tutte} to locally finite graphs and by Richter~\cite{R-Tutte} to graphs of arbitrary degree.
In order to state this result formally, we need some definitions.

Let $G$ be a graph.
A \emph{\td} of~$G$ is a pair $(T,\cV)$ of a tree $T$ and a family $\cV=(V_t)_{t\in T}$ of vertex sets $V_t\sub V(G)$, one for each vertex of~$T$, such that
\begin{enumerate}[(T1)]
\item $V = \bigcup_{t\in T}V_t$;
\item for every $e\in E(G)$ there exists some $t\in V(T)$ such that $V_t$ contains both end vertices of~$e$;
\item $V_{t_1} \cap V_{t_3} \sub V_{t_2}$ whenever $t_2$ lies on the $t_1$--$t_3$ path in~$T$.
\end{enumerate}

The sets $V_t$ are the \emph{parts} of the \td\ $(T,\cV)$ and we also call the graph induced by $V_t$ a \emph{part} of~$(T,\cV)$.
The intersections $V_{t_1}\cap V_{t_2}$ for edges $t_1t_2$ of~$T$ are its \emph{adhesion sets} and the maximum size of the adhesion sets is the \emph{adhesion} of $(T,\cV)$.
For a part $V_t$, its \emph{torso} is the graph whose vertex set is $V_t$ and whose edge set is
\[
\{xy\in E(G)\mid x,y\in V_t\}\ \cup\ \{xy\mid \{x,y\}\sub V_t\text{ is an adhesion set}\}.
\]

There are several ways of constructing \td s.
If its construction commutes with graph isomorphisms\footnote{Note that graph isomorphisms induce isomorphisms between \td s in the following sense: every graph isomorphism $\varphi\colon G\to G'$ induces a canonical bijection between the vertex sets of~$G$ and~$G'$ and for every \td\ $(T,\cV)$ of~$G$, there is a \td\ $(T,\cV')$ of~$G'$ with $\cV'=\{\varphi(V_t)\mid V_t\in\cV\}$.} we call the \td\ \emph{canonical}.

\begin{thm}\cite{R-Tutte}\label{thm_tutte}
Every $2$-connected graph $G$ has a canonical \td\ of adhesion~$2$ each of whose torsos is either $3$-connected or a cycle or a complete graph on two vertices.
\qed
\end{thm}

We call a \td\ as in Theorem~\ref{thm_tutte} a \emph{Tutte decomposition}.
Note that there may be more than one Tutte decomposition even tough it is a canonical \td.

\begin{thm}\label{thm_mainCon2}
Let $G$ be a planar finitely separable $2$-connected graph.
If~$G$ has a Tutte decomposition $(T,\cV)$ whose adhesion sets of $(T,\cV)$ are complete graphs, then $G$ has a canonical nested set of cycles that generates its cycles space.

In addition, we may choose the generating cycles so that each of them lies in a unique part of $(T,\cV)$.
\end{thm}

\begin{proof}
Due to the assumption on the adhesion sets, we know that not only the torsos but also the parts theirselves are either cycles or $3$-connected.
First, we show:
\begin{equation}\label{eq_mainCon2_1}
\begin{minipage}[t]{0.9\textwidth}
\em
Cycles of distinct parts are nested.
\end{minipage}
\end{equation}
Let $\varphi$ be the planar embedding of~$G$ into~$\real^2$ and let $C_1,C_2$ be cycles in distinct parts.
If $C_1$ and $C_2$ are not nested, then both faces of $\real^2\sm \varphi(C_1)$ contain vertices of~$C_2$ and thus $C_1$ and $C_2$ have two common vertices.
As they lie in distinct parts and the intersection of distinct parts lies in some adhesion set, which has size~$2$, the cycles $C_1$ and~$C_2$ have precisely two common vertices $x,y$ and these form an adhesion set of $(T,\cV)$.
Let $V_i$ be the part that contains~$C_i$ for $i=1,2$.
If one of the parts, say $V_1$, is a cycle, then $C_1$ contains the edge $xy$.
As it contains no other vertex of~$C_2$, it cannot contains vertices of both faces of $\real^2\sm \varphi(C_2)$.
Thus, each of the two parts $V_i$ is $3$-connected.
In particular, we find in~$V_1$ a third path~$P$ from every vertex of the bounded face of~$\real^2\sm \varphi(C_2)$ to every vertex of the unbounded face. But then $\varphi(P)$ crosses~$\varphi(C_2)$, so there is a third vertex in $V_1\cap V_2$.
This contradiction shows~(\ref{eq_mainCon2_1}).

Since $(T,\cV)$ is canonical, parts that are cycles are mapped to parts that are cycles.
So if we take the set $\DF_1$ of cycles that appear as a part in $(T,\cV)$, then this set is canonical and nested due to~(\ref{eq_mainCon2_1}).
Let $\HF$ be the set of $3$-connected parts of $(T,\cV)$.
For every $H\in\HF$ we find, by Theorem~\ref{thm_main}, a canonical nested set $\DF_H$ of cycles generating the cycles space of~$H$.
We may also assume that for $H\in\HF$ and $\alpha\in\Aut(G)$, we have $\alpha(\DF_H)=\DF_{\alpha(H)}$.

As $(T,\cV)$ is canonical, the set
\[
\DF_2:=\bigcup_{H\in\HF}\DF_H
\]
is canonical and it is nested by~(\ref{eq_mainCon2_1}).
We claim that $\DF:=\DF_1\cup\DF_2$ satisfies the assertion.
The additional part is obvious by definition of~$\DF$.
As~$\DF$ is nested and canonical, it remains to prove that it generates the cycle space of~$G$.
So let $C$ be any element of the cycles space of~$G$.
It suffices to consider the case that $C$ is a cycle, as all cycles generate the cycle space.
If $C$ lies in a part of $(T,\cV)$, then it is generated by~$\DF$ by defintion of~$\DF$.
If $C$ does not lie in any part of~$(T,\cV)$, then there is some adhesion set $\{x,y\}$ of $(T,\cV)$ that disconnects~$C$.
Let $P_1, P_2$ be the two distinct $x$--$y$ paths on~$C$.
Then $C$ is the sum of $P_1+yx$ and $P_2+yx$.
By induction on the length of~$C$, the cycles $P_1+yx$ and $P_2+yx$ are generated by~$\DF$ and so is~$C$.
\end{proof}

Note that we can embed any planar $2$-connected graph $G$ in a planar $2$-connected graph $G'$ such that $G'$ is obtained from~$G$ by making the adhesion sets of a Tutte decomposition $(T,\cV)$ of~$G$ complete.
This is always possible as any adhesion set separates the graph, so lies on the boundary of some face and disconnects this face boundary; thus, it is not possible that two new such edges `cross'\footnote{This shall mean that adding one of those edges results in a planar graph where the end vertices of the second edge do not lie on any common face boundary.} since no adhesion set of any \td\ can separate any other of its adhesion sets.

As each cycle of a graph lies in a unique maximal $2$-connected subgraph, an immediate consequence of Theorem~\ref{thm_mainCon2} using the well-known block-cutvertex tree and the fact that cycles in distinct components are disjoint is the following:

\begin{thm}\label{thm_mainCon1}
Let $G$ be a planar finitely separable graph such that for each of its maximal $2$-connected subgraphs the adhesion sets of a Tutte decomposition are complete graphs on two vertices.
Then $G$ has a canonical nested set of cycles that generates its cycles space.\qed
\end{thm}

\section*{Acknowledgement}

I thank M.J.~Dunwoody for showing me~\cite{D-Networks} which led to a sharpening of our previous main result.

%I thank J.~Carmesin, K.~Heuer, and J.~Pott for discussions on the topic of this paper.

\bibliographystyle{amsplain}
\bibliography{Bibs}

\providecommand{\bysame}{\leavevmode\hbox to3em{\hrulefill}\thinspace}
\providecommand{\MR}{\relax\ifhmode\unskip\space\fi MR }
% \MRhref is called by the amsart/book/proc definition of \MR.
\providecommand{\MRhref}[2]{%
  \href{http://www.ams.org/mathscinet-getitem?mr=#1}{#2}
}
\providecommand{\href}[2]{#2}
\begin{thebibliography}{10}

\bibitem{BD-Duality}
H.~Bruhn and R.~Diestel, \emph{Duality in infinite graphs}, Comb., Probab.\
  Comput. \textbf{15} (2006), 75--90.

\bibitem{DiDu-GroupsGraphs}
W.~Dicks and M.J. Dunwoody, \emph{Groups {A}cting on {G}raphs}, Cambridge
  Stud.\ Adv.\ Math., vol.~17, Cambridge Univ.\ Press, 1989.

\bibitem{GelbesBuch}
R.~Diestel, \emph{Graph {T}heory \emph{(4th edition)}}, Springer-Verlag, 2010.

\bibitem{DSS-Tutte}
C.~Droms, B.~Servatius, and H.~Servatius, \emph{The structure of locally finite
  two-connected graphs}, Electron.\ J.\ of Comb. \textbf{2} (1995), research
  paper 17.

\bibitem{D-Networks}
M.J. Dunwoody, \emph{Structure trees and networks}, arXiv:1311.3929.

\bibitem{GH-PlanarGroups}
A.~Georgakopoulos and M.~Hamann, \emph{Planar presentations of groups}, in
  preparation.

\bibitem{H-Accessibility}
M.~Hamann, \emph{Accessibility in transitive graphs}, submitted,
  arXiv:1404.7677.

\bibitem{ImWhi}
W.~Imrich, \emph{{On Whitney's theorem on the unique embeddability of
  3-connected planar graphs.}}, {Recent Adv. Graph Theory, Proc. Symp. Prague
  1974}, 1975, pp.~303--306 (English).

\bibitem{R-Tutte}
R.B. Richter, \emph{Decomposing infinite $2$-connected graphs into
  $3$-connected components}, Electron.\ J.\ of Comb. \textbf{11} (2004),
  research paper 25.

\bibitem{T-Duality}
C.~Thomassen, \emph{Duality of infinite graphs}, J.\ Combin.\ Theory (Series B)
  \textbf{33} (1982), no.~2, 137--160.

\bibitem{Tutte}
W.T. Tutte, \emph{Graph {T}heory}, Cambridge University Press, Cambridge, 1984.

\bibitem{whitney_congruent_1932}
H.~Whitney, \emph{Congruent graphs and the connectivity of graphs}, American
  J.\ of Mathematics \textbf{54} (1932), no.~1, 150--168.

\end{thebibliography}

\end{document}